\documentclass[12pt, a4paper, oneside]{amsart}
\usepackage{amsmath, amsthm, amssymb, geometry, mathrsfs, hyperref, tikz-cd}
\usepackage[shortlabels, inline]{enumitem}

\newtheorem{theorem}{Theorem}[section]
\newtheorem{proposition}[theorem]{Proposition}

\newtheorem{lemma}[theorem]{Lemma}
\newtheorem{corollary}[theorem]{Corollary}

\theoremstyle{definition}
\newtheorem{definition}[theorem]{Definition}

\newtheorem{problem}[theorem]{Problem}

\newtheorem{remark}[theorem]{Remark}

\newcommand\B{\mathbb{B}}
\newcommand\C{\mathbb{C}}
\newcommand\D{\mathbb{D}}

\newcommand\N{\mathbb{N}}
\renewcommand\P{\mathbb{P}}

\newcommand\R{\mathbb{R}}
\newcommand\Z{\mathbb{Z}}

\newcommand{\cO}{\mathcal{O}}

\newcommand{\id}{\mathrm{id}}
\newcommand{\pr}{\mathrm{pr}}

\DeclareMathOperator{\Aut}{Aut}

\DeclareMathOperator{\Op}{Op}
\DeclareMathOperator{\Supp}{Supp}

\newcommand{\CAP}{\mathrm{CAP}}
\newcommand{\CIP}{\mathrm{CIP}}
\newcommand{\BOPA}{\mathrm{BOPA}}
\newcommand{\BOPI}{\mathrm{BOPI}}
\newcommand{\BOPJI}{\mathrm{BOPJI}}
\newcommand{\BOPAI}{\mathrm{BOPAI}}
\newcommand{\BOPAJI}{\mathrm{BOPAJI}}
\newcommand{\PCAP}{\mathrm{PCAP}}
\newcommand{\PCIP}{\mathrm{PCIP}}
\newcommand{\POPA}{\mathrm{POPA}}
\newcommand{\POPI}{\mathrm{POPI}}
\newcommand{\POPJI}{\mathrm{POPJI}}
\newcommand{\POPAI}{\mathrm{POPAI}}
\newcommand{\POPAJI}{\mathrm{POPAJI}}

\newcommand{\HAP}{\mathrm{HAP}}
\newcommand{\Ell}{\mathrm{Ell}}

\newcommand{\POPAIsec}{\POPAI^{\mathrm{sec}}}
\newcommand{\HAPsec}{\HAP^{\mathrm{sec}}}

\begin{document}

\title[Elliptic characterization and unification of Oka maps]{Elliptic characterization and \\ unification of Oka maps}
\author{Yuta Kusakabe}
\address{Department of Mathematics, Graduate School of Science, Osaka University, Toyonaka, Osaka 560-0043, Japan}
\email{y-kusakabe@cr.math.sci.osaka-u.ac.jp}

\begin{abstract}
We generalize our elliptic characterization of Oka manifolds to Oka maps.
The generalized characterization can be considered as an affirmative answer to the relative version of Gromov's conjecture.
As an application, we unify previously known Oka principles for submersions; namely the Gromov type Oka principle for subelliptic submersions and the Forstneri\v{c} type Oka principle for holomorphic fiber bundles with $\CAP$ fibers.
We also establish the localization principle for Oka maps which gives new examples of Oka maps.
\end{abstract}

\maketitle

%
%

\section{Introduction}

In 1989, Gromov's seminal paper \cite{Gromov1989} on the Oka principle initiated modern Oka theory.
Forstneri\v{c}, L\'{a}russon and others developed it into the theory of Oka manifolds and Oka maps (cf. \cite{Forstneric2013,Forstneric2017,Larusson2004}).
One of the most fundamental problems is to determine which holomorphic submersions enjoy the following Oka property.
Here, for a subset $A$ of a topological space $X$ we denote by $\Op A=\Op_{X}A$ a non-specified open neighborhood of $A$ in $X$.

\begin{definition}
\label{definition:oka_map}
(1) A holomorphic submersion $\pi:Y\to S$ between (reduced) complex spaces enjoys \emph{POPAI (the Parametric Oka Property with Approximation and Interpolation)} if for any Stein space $X$, any closed complex subvariety $X'\subset X$, any compact $\cO(X)$-convex subset $K\subset X$, any compact Hausdorff spaces $P_{0}\subset P$, any (continuous) family of holomorphic maps $F:P\times X\to S$ and any family of continuous maps $f_{0}:P\times X\to Y$ such that
\begin{enumerate}[(a)]
	\item $\pi\circ f_{0}=F$, and
	\item $f_{0}|_{P_{0}\times X}$, $f_{0}|_{P\times X'}$ and $f_{0}|_{P\times\Op K}$ are families of holomorphic maps,
\end{enumerate}
there exists a homotopy $f_{t}:P\times X\to Y,\ t\in[0,1]$ such that the following hold for each $t\in[0,1]$;
\begin{enumerate}[(i)]
	\item $\pi\circ f_{t}=F$,
	\item $f_{t}=f_{0}$ on $(P_{0}\times X)\cup(P\times X')$,
	\item $f_{t}|_{P\times\Op K}$ is a family of holomorphic maps which approximates $f_{0}$ uniformly on $P\times K$, and
	\item $f_{1}:P\times X\to Y$ is a family of holomorphic maps.
\end{enumerate}
(2) A holomorphic submersion is an \emph{Oka map} if it is a (topological) fibration\footnote{By a topological fibration, we mean a Serre fibration or a Hurewicz fibration.
Since complex spaces admit triangulations, these notions are equivalent for continuous maps between complex spaces (cf. \cite{Arnold1972}).} and enjoys $\POPAI$.
\end{definition}

There are mainly two types of Oka principles.
One is the Gromov type Oka principle \cite{Gromov1989,Forstneric2002a,Forstneric2010} which states that every subelliptic submersion enjoys POPAI.
The other is the Forstneri\v{c} type Oka principle \cite{Forstneric2009,Forstneric2010} which shows $\POPAI$\footnote{Here, the parameter spaces $P_0\subset P$ in Definition \ref{definition:oka_map} are restricted to be Euclidean (cf. \cite[\S 7.4]{Forstneric2017}). This restriction will be removed later (see Corollary \ref{corollary:parameter}).} of a holomorphic fiber bundle whose fiber enjoys \emph{CAP} (the \emph{Convex Approximation Property}, see Definition \ref{definition:CAP}).
These Oka principles are known to be mutually independent.
Namely, there exists a subelliptic submersion which is not locally trivial at any base point (e.g. a complete family of complex tori \cite[Theorem 16]{Larusson2012}), and there exists a non-subelliptic holomorphic fiber bundle with a $\CAP$ fiber (\cite[Corollary 3.2]{Kusakabe2020}).
Moreover, there is a holomorphic submersion which enjoys POPAI but is neither subelliptic nor locally trivial at any base point (Proposition \ref{proposition:example}).
Thus it is natural to ask whether there exists a characterization of POPAI which implies these Oka principles.

In the present paper, we characterize $\POPAI$ by the following \emph{convex ellipticity} which is a variant of Gromov's Condition $\Ell_{1}$ \cite[p.\,71]{Gromov1986}.

\begin{definition}
\label{definition:cell}
A holomorphic submersion $\pi:Y\to S$ (between complex spaces) is \emph{convexly elliptic} if there exists an open cover $\{U_{\alpha}\}_{\alpha}$ of $S$ such that for any $n\in\N$, any compact convex set $K\subset\C^{n}$ and any holomorphic map $f:\Op K\to Y$ with $f(K)\subset \pi^{-1}(U_{\alpha})$ for some $\alpha$ there exists a holomorphic map $s:\Op K\times\C^{N}\to Y$ such that
\begin{enumerate}
	\item $s(z,0)=f(z)$, $\pi\circ s(z,w)=\pi\circ f(z)$ for all $(z,w)\in\Op K\times\C^{N}$, and
	\item $s(z,\cdot):\C^{N}\to\pi^{-1}(\pi\circ f(z))$ is a submersion at $0$ for each $z\in K$.
\end{enumerate}
\end{definition}

Without using the above Oka principles, we can easily prove that every subelliptic submersion is convexly elliptic (cf. \cite[Proof of Proposition 8.8.11\,(b)]{Forstneric2017}) and that every holomorphic fiber bundle with a $\CAP$ fiber is convexly elliptic (Remark \ref{remark:CAP} and Proposition \ref{proposition:CAP=>CEll}).
Thus the following main theorem unifies them.

\begin{theorem}
\label{theorem:characterization}
A holomorphic submersion enjoys $\POPAI$ if and only if it is convexly elliptic.
In particular, a holomorphic submersion is an Oka map if and only if it is a convexly elliptic fibration.
\end{theorem}

Since a complex manifold $Y$ is an \emph{Oka manifold} if and only if the constant submersion $Y\to *$ is an Oka map, Theorem \ref{theorem:characterization} generalizes our previous characterization of Oka manifolds \cite[Theorem 2.2]{Kusakabe2019}.
Recall that the previous characterization gave an affirmative answer to Gromov's conjecture \cite[\S 1.4.E$''$]{Gromov1989} which essentially states that the Oka property and Gromov's Condition $\Ell_{1}$ are equivalent for manifolds (see \cite[\S4.2]{Kusakabe2019}).
Therefore Theorem \ref{theorem:characterization} can be considered as an affirmative answer to the relative version of Gromov's conjecture.

Theorem \ref{theorem:characterization} has various applications.
In the same way as in \cite{Kusakabe2019}, it implies the localization principle for Oka maps (Corollary \ref{corollary:localization}) which gives new examples of Oka maps.
We also obtain the equivalences between fourteen Oka properties of a submersion (Corollary \ref{corollary:equivalence}), the invariance of $\POPAI$ (Corollary \ref{corollary:fibration}) and the refinements of previously known Oka principles (Corollary \ref{corollary:parameter} and Corollary \ref{corollary:dimensionwise}).

This paper is organized as follows.
In Section \ref{section:sprays}, we recall the notion of dominating sprays which plays a fundamental role in the proof of Theorem \ref{theorem:characterization}.
The basic properties of dominating sprays are also reviewed.
Before proving Theorem \ref{theorem:characterization}, we first establish the Oka principle for sections of stratified convexly elliptic submersions in Section \ref{section:sections}.
By using this Oka principle, we prove the stratified version of Theorem \ref{theorem:characterization} in Section \ref{section:lifts}.
In Section \ref{section:applications}, we give applications and remarks.

%
%

\section{Dominating sprays}
\label{section:sprays}

In this section, we recall the notion of dominating sprays and a few facts about it.
We denote by $\pr_{X_{\lambda_{0}}}:\prod_{\lambda\in\Lambda}X_{\lambda}\to X_{\lambda_{0}}$ the projection map to $X_{\lambda_{0}}\ (\lambda_{0}\in\Lambda)$.

\begin{definition}
\label{definition:spray}
Let $X$ be a complex space, $\pi:Y\to S$ be a holomorphic submersion, $A\subset X$ be a subset and $P$ be a topological space.
\begin{enumerate}[leftmargin=*]
\item A \emph{(local) $\pi$-spray} over a family of holomorphic maps $f:P\times\Op A\to Y$ is a family of holomorphic maps $s:P\times\Op A\times W\to Y$ where $W\subset\C^N$ is an open neighborhood of $0$ such that $s(p,x,0)=f(p,x)$, $\pi\circ s(p,x,w)=\pi\circ f(p,x)$ for all $(p,x,w)\in P\times\Op A\times W$.
Particularly in the case of $W=\C^{N}$, $s$ is also called a \emph{global $\pi$-spray}.
\item A $\pi$-spray $s:P\times\Op A\times W\to Y$ is \emph{dominating} if $s(p,x,\cdot):\C^{N}\to\pi^{-1}(\pi\circ s(p,x,0))$ is a submersion at $0$ for each $(p,x)\in P\times A$.
\end{enumerate}
\end{definition}

Note that the holomorphic map $s:\Op K\times\C^{N}\to Y$ in Definition \ref{definition:cell} is nothing but a dominating global $\pi$-spray over $f:\Op K\to Y$.

The following fact ensures the existence of dominating local sprays.
Here, a compact subset of a complex space is called a \emph{Stein compact} if it admits a basis of open Stein neighborhoods.

\begin{lemma}[{cf. \cite[Lemma 5.10.4 and p.\,254]{Forstneric2017}}]
\label{lemma:local_spray}
Let $\pi:Y\to S$ be a holomorphic submersion, $K$ be a Stein compact in a complex space and $P$ be a compact Hausdorff space.
Then for any family of holomorphic maps $f:P\times\Op K\to Y$ there exist an open neighborhood $W\subset\C^{N}$ of $0$ and a dominating $\pi$-spray $P\times\Op K\times W\to Y$ over $f$.
\end{lemma}

The next fact states that the dominability of a spray is equivalent to the existence of a right inverse of the associated fiber preserving map.
This allows us to transfer nonlinear problems to linear problems.
We denote by $\Gamma_{f}$ the graph of $f:P\times\Op K\to Y$ in the pullback $(\pi\circ f)^{*}Y\subset P\times\Op K\times Y$.

\begin{lemma}[{cf. \cite[p.\,254]{Forstneric2017}, \cite[Lemma 2.5]{Kusakabe2019}}]
\label{lemma:right_inverse}
Let $\pi:Y\to S$ be a holomorphic submersion, $K$ be a Stein compact in a complex space, $P$ be a compact Hausdorff space and $f:P\times\Op K\to Y$ be a family of holomorphic maps.
Then a $\pi$-spray $s:P\times\Op K\times W\to Y$ over $f$ is dominating if and only if there exists a fiber preserving (with respect to $\pr_{P\times\Op K}$) continuous map $\iota:\Op_{(\pi\circ f)^{*}Y}\Gamma_{f}\to P\times\Op K\times W$ which satisfies $(\pr_{P\times\Op K},s)\circ\iota=\id,\, \pr_{W}\circ\iota|_{\Gamma_{f}}\equiv 0$ and defines a holomorphic map $\iota(p,\cdot,\cdot):\Op\Gamma_{f(p,\cdot)}\to\Op K\times W$ for each fixed $p\in P$.
\end{lemma}

%
%

\section{Oka principle for sections}
\label{section:sections}

Before proving the Oka principle for lifts (Theorem \ref{theorem:characterization}), we first establish the Oka principle for sections.
Let us introduce the following terminology (compare with Definition \ref{definition:oka_map}).

\begin{definition}
\label{definition:POP}
A holomorphic submersion $h:Z\to X$ onto a Stein space enjoys $\POPAIsec$ ($\POPAI$ for sections) if for any closed complex subvariety $X'\subset X$, any compact $\cO(X)$-convex subset $K\subset X$, any compact Hausdorff spaces $P_{0}\subset P$ and any family of continuous maps $f_{0}:P\times X\to Z$ such that
\begin{enumerate}
	\item $h\circ f_{0}=\pr_{X}$, and
	\item $f_{0}|_{P_{0}\times X}$, $f_{0}|_{P\times X'}$ and $f_{0}|_{P\times\Op K}$ are families of holomorphic maps,
\end{enumerate}
there exists a homotopy $f_{t}:P\times X\to Z,\ t\in[0,1]$ such that the following hold for each $t\in[0,1]$;
\begin{enumerate}[(a)]
	\item $h\circ f_{t}=\pr_{X}$,
	\item $f_{t}=f_{0}$ on $(P_{0}\times X)\cup(P\times X')$,
	\item $f_{t}|_{P\times\Op K}$ is a family of holomorphic maps which approximates $f_{0}$ uniformly on $P\times K$, and
	\item $f_{1}:P\times X\to Z$ is a family of holomorphic maps.
\end{enumerate}
\end{definition}

Our goal in this section is to prove the following Oka principle.
Here, $\B^{N}$ is the unit ball in $\C^{N}$.

\begin{theorem}
\label{theorem:oka_principle_for_sections}
Let $h:Z\to X$ be a holomorphic submersion onto a Stein space $X$.
Assume that there exists a stratification of $X$ by closed complex subvarieties $X=X_{0}\supset X_{1}\supset\cdots\supset X_{m}=\emptyset$ such that the following hold for each $k=0,\ldots,m-1$;
\begin{enumerate}
\item the stratum $\Sigma_{k}=X_{k}\setminus X_{k+1}$ is smooth, and
\item there exists an open cover $\{U_{\alpha}\}_{\alpha}$ of $\Sigma_{k}$ such that for any $N\in\N$, any compact set $K\subset U_{\alpha}\times\B^{N}$ which is convex in some holomorphic coordinate system and any holomorphic section $f:\Op_{U_{\alpha}\times\B^{N}}K\to h^{-1}(U_{\alpha})\times\B^{N}$ of $h\times\id_{\B^{N}}:h^{-1}(U_{\alpha})\times\B^{N}\to U_{\alpha}\times\B^{N}$ there exists a dominating global $(h\times\id_{\B^{N}})$-spray $\Op_{U_{\alpha}\times\B^{N}}K\times\C^{L}\to h^{-1}(U_{\alpha})\times\B^{N}$ over $f$.
\end{enumerate}
Then $h$ enjoys $\POPAIsec$.
\end{theorem}

\begin{remark}
\label{remark:pullback}
A holomorphic submersion $\pi:Y\to S$ is said to be \emph{stratified convexly elliptic} if there exists a stratification of $S$ by closed complex subvarieties such that on each stratum $\Sigma$ the restriction $\pi^{-1}(\Sigma)\to \Sigma$ is convexly elliptic.
For such a submersion $\pi:Y\to S$ and any holomorphic map $F:X\to S$ from a Stein space $X$ with $F(X)\subset\pi(Y)$, the pullback submersion $h=F^{*}\pi:F^{*}Y\to X$ satisfies the assumption in Theorem \ref{theorem:oka_principle_for_sections}.
This will be used in the proof of the Oka principle for lifts (see the proof of Lemma \ref{lemma:parametric_spray}).
\end{remark}

%
%

\subsection{Homotopy approximation property: Reduction of Theorem \ref{theorem:oka_principle_for_sections}}

In order to prove Theorem \ref{theorem:oka_principle_for_sections}, we need to recall the tools used in the proofs of the previously known Oka principles.
In the proof of the Gromov type Oka principle, the following axiom plays a fundamental role.
It is used to prove the Heftungslemma for sections of subelliptic submersions (cf. \cite[Proposition 6.7.2]{Forstneric2017}).

\begin{definition}[{cf. \cite[Definition 6.6.5]{Forstneric2017}, \cite[Proposition 2.1]{Forstneric2010a}}]
\label{definition:HAP}
A holomorphic submersion $h:Z\to X$ enjoys $\HAPsec$ (the \emph{Homotopy Approximation Property  for sections}) if for  any Stein compact $L\subset X$, any compact $\cO(L)$-convex subset $K\subset L$, any compact Hausdorff spaces $P_{0}\subset P$ and any family of continuous maps $f_{0}:Q\times\Op L\to Z$ where $Q=P\times[0,1],\ Q_{0}=(P\times\{0\})\cup(P_{0}\times[0,1])$ such that
	\begin{enumerate}
	\item $h\circ f_{0}=\pr_{\Op L}$, and
	\item $f_{0}|_{Q_{0}\times\Op L}$ and $f_{0}|_{Q\times\Op K}$ are families of holomorphic maps,
	\end{enumerate}
there exists a homotopy $f_{t}:Q\times\Op L\to Z,\ t\in[0,1]$ such that the following hold for each $t\in[0,1]$;
	\begin{enumerate}[(a)]
	\item $h\circ f_{t}=\pr_{\Op L}$,
	\item $f_{t}=f_{0}$ on $Q_{0}\times\Op L$,
	\item $f_{t}|_{Q\times\Op K}$ is a family of holomorphic maps which approximates $f_{0}$ uniformly on $Q\times K$, and
	\item $f_{1}:Q\times\Op L\to Z$ is a family of holomorphic maps.
	\end{enumerate}
\end{definition}

The following is one of the crucial steps in the proof of the Gromov type Oka principle.
Recently, this kind of axiomatization of the proof of an Oka principle is generalized to the case of sheaves by Studer \cite{Studer2020}\footnote{Instead of $\HAPsec$, he used the axiom called \emph{weak flexibility} to prove his Oka principle, which is an axiomatization of the Heftungslemma \cite[Proposition 6.7.2]{Forstneric2017}.}.

\begin{theorem}[{cf. \cite[Theorem 6.6.6 and Proof of Proposition 6.7.2]{Forstneric2017}}] 
\label{theorem:HAP=>POP}
Let $h:Z\to X$ be a stratified holomorphic submersion onto a Stein space $X$.
Assume that for each stratum $\Sigma\subset X$ there exists an open cover $\{U_{\alpha}\}_{\alpha}$ of $\Sigma$ such that $h\times\id_{\B^{N}}:h^{-1}(U_{\alpha})\times\B^{N}\to U_{\alpha}\times\B^{N}$ enjoys $\HAPsec$ for each $N\in\N$ .
Then $h$ enjoys $\POPAIsec$.
\end{theorem}

\begin{remark}
\label{remark:HAP}
In \cite[Theorem 6.6.6]{Forstneric2017}, it is only assumed that $h:h^{-1}(U_{\alpha})\to U_{\alpha}$ enjoys $\HAPsec$ in the above situation (i.e. $\B^{N}$ does not appear).
In the proof of the Heftungslemma \cite[Proposition 6.7.2]{Forstneric2017} which is used to prove $\POPAIsec$, however, we need to approximate dominating local sprays over sections.
This is the reason why we need to assume that $h\times\id_{\B^{N}}:h^{-1}(U_{\alpha})\times\B^{N}\to U_{\alpha}\times\B^{N}$ enjoys $\HAPsec$ for each $N\in\N$.
\end{remark}

Thanks to Theorem \ref{theorem:HAP=>POP}, Theorem \ref{theorem:oka_principle_for_sections} is reduced to the following proposition.

\begin{proposition}
\label{proposition:reduction}
Let $h:Z\to X$ be a holomorphic submersion onto a complex manifold $X$.
Assume that for any compact set $K\subset X$ which is convex in some holomorphic coordinate system and any holomorphic section $f:\Op K\to Z$ there exists a dominating global $h$-spray $\Op K\times\C^{N}\to Z$ over $f$.
Then $h$ enjoys $\HAPsec$.
\end{proposition}

To simplify the approximation problem further, we use the induction used in the proof of the Forstneri\v{c} type Oka principle.
In the proof \cite[\S 5.7--5.13]{Forstneric2017}, the desired homotopy of families of sections is constructed inductively on an increasing exhausting sequence of compact strongly pseudoconvex domains (in fact, we need to retake an increasing sequence when we extend the homotopy across a totally real disc but this does not matter; see \cite[\S 5.11]{Forstneric2017}).

\begin{remark}
In the situation of Definition \ref{definition:HAP}, let us assume further that we have a compact subset $K\subset K'\subset\Op L$ and a homotopy $f_{t}:Q\times\Op K'\to Z,\ t\in[0,1]$ which satisfies (a)--(d) where $\Op L$ is replaced by $\Op K'$.
Then by using a continuous function $\chi:\Op L\to[0,1]$ such that $\chi|_{\Op K'}\equiv 1$ and $\Supp\chi$ is contained in a small neighborhood of $K'$, we can extend $f_{t}$ to the homotopy $f_{t}:Q\times\Op L\to Z,\ t\in[0,1]$ by $(q,x)\mapsto f_{\chi(x)t}(q,x)$.
This homotopy satisfies (a)--(c) and $f_{1}|_{Q\times\Op K'}$ is a family of holomorphic maps.
\end{remark}

By the above remark and Forstneri\v{c}'s induction, it suffices to consider the approximation problem only around each compact strongly pseudoconvex domain which appears in the exhausting sequence of a Stein manifold $\Op L$.
In the exhausting sequence $\cdots\subset K_{j}\subset K_{j+1}\subset\cdots$ used in Forstneri\v{c}'s induction, for each $j$ there exists a \emph{special Cartan pair} $(A,B)$ such that $K_{j}=A$ and $K_{j+1}=A\cup B$ (see \cite[\S5.10]{Forstneric2017}).
Let us recall the definition of Cartan pairs.

\begin{definition}[{cf. \cite[Definition 2.4]{Forstneric2019}}]
\label{definition:cartan_pair}
\ 
\begin{enumerate}[leftmargin=*]
	\item A pair $(A,B)$ of compact sets in a complex space is a \emph{Cartan pair} if each of $A,B,A\cap B$ and $A\cup B$ is a Stein compact and the separation condition $\overline{A\setminus B}\cap\overline{B\setminus A}=\emptyset$ is satisfied.
	\item A Cartan pair $(A,B)$ in a complex manifold is \emph{special} if there exists a holomorphic coordinate system on $\Op B$ in which $B$ and $A\cap B$ are regular\footnote{A compact set in a topological space is \emph{regular} if it coincides with the closure of its interior.} convex sets.
	\item A special Cartan pair $(A,B)$ is \emph{very special} if there exists a holomorphic coordinate system on $\Op B$ in which $B$ is convex and $A\cap B$ is the intersection of $B$ and a closed half space.
\end{enumerate}
\end{definition}

By using the following lemma of Forstneri\v{c}, we can refine the exhausting sequence in the induction by very special Cartan pairs.

\begin{lemma}[{cf. \cite[Lemma 2.5]{Forstneric2019}}]
\label{lemma:very_special}
Let $(A,B)$ be a special Cartan pair in a complex manifold $X$.
Then for any open neighborhood $U\subset X$ of $A$ there exists a finite sequence
\begin{align*}
A\subset A_{1}\subset A_{2}\subset\cdots\subset A_{k}=A\cup B
\end{align*}
of compact sets such that $A_{1}\subset U$ and $A_{j+1}=A_{j}\cup B_{j}\ (j=1,\ldots,k-1)$ where each $(A_{j},B_{j})$ is a very special Cartan pair.
\end{lemma}

By the previous reduction to Proposition \ref{proposition:reduction} and the above arguments, Theorem \ref{theorem:oka_principle_for_sections} is further reduced to the following proposition.
It is the goal of the next subsection.

\begin{proposition}
\label{proposition:reduction'}
Let $h:Z\to X$ be a holomorphic submersion onto a complex manifold $X$ and $(A,B)$ be a very special Cartan pair in $X$.
Assume that for any holomorphic section $f:\Op(A\cap B)\to Z$ there exists a dominating global $h$-spray $\Op(A\cap B)\times\C^{N}\to Z$ over $f$.
Then for any compact Hausdorff spaces $P_{0}\subset P$ and any family of continuous maps $f_{0}:Q\times\Op(A\cup B)\to Z$ where $Q=P\times[0,1],\ Q_{0}=(P\times\{0\})\cup(P_{0}\times[0,1])$ such that
	\begin{enumerate}
	\item $h\circ f_{0}=\pr_{\Op(A\cup B)}$, and
	\item $f_{0}|_{Q_{0}\times\Op(A\cup B)}$ and $f_{0}|_{Q\times\Op A}$ are families of holomorphic maps,
	\end{enumerate}
there exists a homotopy $f_{t}:Q\times\Op(A\cup B)\to Z,\ t\in[0,1]$ such that the following hold for each $t\in[0,1]$;
	\begin{enumerate}[(a)]
	\item $h\circ f_{t}=\pr_{\Op(A\cup B)}$,
	\item $f_{t}=f_{0}$ on $Q_{0}\times\Op(A\cup B)$,
	\item $f_{t}|_{Q\times\Op A}$ is a family of holomorphic maps which approximates $f_{0}$ uniformly on $Q\times A$, and
	\item $f_{1}:Q\times\Op(A\cup B)\to Z$ is a family of holomorphic maps.
	\end{enumerate}
\end{proposition}

%
%

\subsection{Approximation theorem: Proof of Proposition \ref{proposition:reduction'}}
\label{subsection:approximation}

Proposition \ref{proposition:reduction'} is an immediate corollary of the following approximation theorem.
It generalizes the approximation theorem \cite[Theorem 2.7]{Kusakabe2019} which was used to prove the implication from convex ellipticity to $\CAP$ for manifolds.

\begin{theorem}
\label{theorem:approximation}
Let $h:Z\to X$ be a holomorphic submersion onto a complex manifold $X$, $(A,B)$ be a very special Cartan pair in $X$, $P_{0}\subset P$ be compact Hausdorff spaces and $f_{0}:Q\times\Op(A\cup B)\to Z$ be a family of continuous maps where $Q=P\times[0,1],\ Q_{0}=(P\times\{0\})\cup(P_{0}\times[0,1])$ such that
	\begin{enumerate}
	\item $h\circ f_{0}=\pr_{\Op(A\cup B)}$, and
	\item $f_{0}|_{Q_{0}\times\Op(A\cup B)}$ and $f_{0}|_{Q\times\Op A}$ are families of holomorphic maps.
	\end{enumerate}
Assume that for each $q\in Q$ there exists a dominating global $h$-spray $\Op(A\cap B)\times\C^N\to Z$ over $f_0(q,\cdot)|_{\Op(A\cap B)}$.
Then there exists a homotopy $f_{t}:Q\times\Op(A\cup B)\to Z,\ t\in[0,1]$ such that the following hold for each $t\in[0,1]$;
	\begin{enumerate}[(a)]
	\item $h\circ f_{t}=\pr_{\Op(A\cup B)}$,
	\item $f_{t}=f_{0}$ on $Q_{0}\times\Op(A\cup B)$,
	\item $f_{t}|_{Q\times\Op A}$ is a family of holomorphic maps which approximates $f_{0}$ uniformly on $Q\times A$, and
	\item $f_{1}:Q\times\Op(A\cup B)\to Z$ is a family of holomorphic maps.
	\end{enumerate}
\end{theorem}

In order to prove the above theorem, we need the following lemma which was proved implicitly in \cite[Proof of Theorem 2.2]{Kusakabe2019}.

\begin{lemma}
\label{lemma:cartan_prime}
Let $(A,B)$ be a very special Cartan pair in a complex manifold $X$.
Then for any open neighborhood $U\subset X$ of $A\cap B$ there exists a Cartan pair $(A',B')$ in $X$ such that
\begin{enumerate}
\item $A\cap B\subset A'\subset U$,
\item $A'\cup B'=B$,
\item $(A\cap B)\cap(A'\cap B')=\emptyset$, and
\item $(A\cap B)\cup(A'\cap B')$ is $\cO(A')$-convex.
\end{enumerate}
\end{lemma}

\begin{proof}
Since $(A,B)$ is a very special Cartan pair, there exist a holomorphic coordinate map $\varphi:\Op B\to\C^{n}$ and an affine $\R$-linear function $\lambda:\C^{n}\to\R$ such that $\varphi(B)$ is convex and $\varphi(A\cap B)=\{z\in\varphi(B):\lambda(z)\leq 0\}$.
Note that there exists $\varepsilon>0$ such that $\varphi^{-1}(\{z\in\varphi(B):\lambda(z)\leq 2\varepsilon\})\subset U$.
Then $A'=\varphi^{-1}(\{z\in\varphi(B):\lambda(z)\leq 2\varepsilon\})$ and $B'=\varphi^{-1}(\{z\in\varphi(B):\lambda(z)\geq\varepsilon\})$ satisfy the desired properties.
\end{proof}

\begin{proof}[Proof of Theorem \ref{theorem:approximation}]
By assumption, for each fixed $q=(p,\tau)\in Q$ there exists a dominating global $h$-spray $s':\Op (A\cap B)\times\C^{N}\to Z$ over $f_{0}(q,\cdot)|_{\Op(A\cap B)}$.
Since $s'$ is dominating, there exists a family of holomorphic maps $\varphi:\Op\,\{p\}\times\Op(A\cap B)\to\C^{N}$ such that $s'\circ(\pr_{\Op(A\cap B)},\varphi)=f_{0}(\cdot,\tau,\cdot)|_{\Op\,\{p\}\times\Op(A\cap B)}$ and $\varphi(p,\cdot)\equiv 0$ by Lemma \ref{lemma:right_inverse}.
Then
\begin{align*}
s=s'\circ(\pr_{\Op(A\cap B)},\pr_{\C^{N}}+\varphi\circ\pr_{\Op\,\{p\}\times\Op (A\cap B)}):\Op\,\{p\}\times\Op (A\cap B)\times\C^{N}\to Z
\end{align*}
is a dominating global $h$-spray over $f_{0}(\cdot,\tau,\cdot)|_{\Op\,\{p\}\times\Op (A\cap B)}$.
By Lemma \ref{lemma:right_inverse} again, there exists a family of holomorphic maps $\tilde f_{0}:\Op\,\{p\}\times\Op\,\{\tau\}\times\Op(A\cap B)\to\C^{N}$ such that $s\circ(\pr_{\Op\,\{p\}\times\Op(A\cap B)},\tilde f_{0})=f_{0}|_{\Op\,\{p\}\times\Op\,\{\tau\}\times\Op(A\cap B)}$ and $\tilde f_{0}(\cdot,\tau,\cdot)\equiv 0$.
Then $s\circ(\pr_{\Op\,\{p\}\times\Op(A\cap B)},\pr_{\C^{N}}+\tilde f_{0}\circ\pr_{\Op\,\{p\}\times\Op\,\{\tau\}\times\Op (A\cap B)})$ defines a dominating global $h$-spray $\Op\,\{p\}\times\Op\,\{\tau\}\times\Op (A\cap B)\times\C^{N}\to Z$ over $f_{0}|_{\Op\,\{p\}\times\Op\,\{\tau\}\times\Op(A\cap B)}$.
By using the above argument and the compactness of $P$, we can find
\begin{enumerate}[(i)]
	\item a cover of $P$ by compact subsets $P^{1},\ldots,P^{l}$,
	\item numbers $0=\tau^{0}<\tau^{1}<\cdots<\tau^{m}=1$,
	\item dominating global $h$-sprays $s^{j,k}:\Op P^j\times\Op (A\cap B)\times\C^{N}\to Z\ (j=1,\ldots,l,\ k=1,\ldots,m)$ over $f_0(\cdot,\tau^{k-1},\cdot)|_{\Op P^{j}\times\Op (A\cap B)}$, and
	\item families of holomorphic maps $\tilde f_{0}^{j,k}:Q^{j,k}\times\Op(A\cap B)\to\C^{N}$ such that $s^{j,k}\circ(\pr_{\Op P^{j}\times\Op(A\cap B)},\pr_{\C^{N}}+\tilde f_{0}^{j,k}\circ\pr_{Q^{j,k}\times\Op(A\cap B)}):Q^{j,k}\times\Op(A\cap B)\times\C^{N}\to Z$ are dominating global $h$-sprays over $f_0|_{Q^{j,k}\times\Op (A\cap B)}$ where $Q^{j,k}=\Op P^{j}\times\Op_{[\tau^{k-1},1]}[\tau^{k-1},\tau^{k}]$.
\end{enumerate}
Note that we may take common $N$ because we are free to increase the rank of a dominating spray (but this fact is not important in the proof).

In the following, we use the so-called \emph{stepwise extension method} (cf. \cite{Forstneric2010}).
We first deform $f_{0}$ on $Q^{1,1}\times\Op (A\cup B)$.
By \cite[Proposition 4.4]{Forstneric2010}, we may assume that $f_{0}|_{(\Op P_{0}\times[0,1])\times\Op (A\cup B)}$ is a family of holomorphic maps from the beginning.
Take an open neighborhood $U\subset X$ of $A\cap B$ such that $s^{1,1}:\Op P^{1}\times U\times\C^{N}\to Z$ and $\tilde f_{0}^{1,1}:Q^{1,1}\times U\to\C^{N}$ are defined and the $h$-spray $Q^{1,1}\times U\times\C^{N}\to Z$ in the above condition (iv) is dominating.
Since $(A,B)$ is a very special Cartan pair, we can take a Cartan pair $(A',B')$ which satisfies (1)--(4) in Lemma \ref{lemma:cartan_prime}.
Set
\begin{align*}
Q_{0}^{1,1}=Q^{1,1}\cap((P\times\{0\})\cup(\Op P_{0}\times[0,1])).
\end{align*}
Take a dominating $h$-spray $s'_{B'}:Q_{0}^{1,1}\times\Op B'\times W\to Z$ over $f_{0}|_{Q_{0}^{1,1}\times\Op B'}$ by using Lemma \ref{lemma:local_spray}.
By shrinking $W\ni 0$, we may assume that $W$ is convex and there exists a spray (with respect to the constant submersion) $\tilde s_{B'}':Q_{0}^{1,1}\times\Op(A'\cap B')\times W\to\C^{N}$ over $0$ such that
\begin{align*}
s^{1,1}\circ(\pr_{\Op P^{1}\times\Op(A'\cap B')},\tilde s_{B'}'+\tilde f_{0}^{1,1}\circ\pr_{Q^{1,1}_{0}\times\Op(A'\cap B')})=s_{B'}'|_{Q_{0}^{1,1}\times\Op(A'\cap B')\times W}
\end{align*}
by the above condition (iv) and Lemma \ref{lemma:right_inverse}.
Take a continuous function $\chi:\Op P^{1}\to[0,1]$ such that $\chi|_{\Op P_{0}\cap\Op  P^{1}}\equiv 1$ and $\Supp\chi$ is contained in a small neighborhood of $P_{0}\cap\Op P^{1}$.
By using this function, let us consider $\rho:Q^{1,1}\to Q_{0}^{1,1},\ (p,\tau)\mapsto (p,\chi(p)\tau)$.
Note that $\rho$ and $\id_{Q^{1,1}}$ are homotopic relative to $Q^{1,1}_{0}$.
Then we can consider the dominating $h$-spray $s_{B'}=s'_{B'}\circ(\rho\times\id_{\Op B'\times W}):Q^{1,1}\times\Op B'\times W\to Z$ and the spray $\tilde s_{B'}=(\tilde s_{B'}'+\tilde f_{0}^{1,1}\circ\pr_{Q^{1,1}_{0}\times\Op(A'\cap B')})\circ(\rho\times\id_{\Op(A'\cap B')\times W}):Q^{1,1}\times\Op(A'\cap B')\times W\to\C^{N}$.
These sprays satisfy
\begin{itemize}
\item $s^{1,1}(\pr_{\Op P^{1}\times\Op(A'\cap B')},\tilde s_{B'})=s_{B'}|_{Q^{1,1}\times\Op(A'\cap B')\times W}$, and
\item $\tilde s_{B'}$ is homotopic to $\tilde f_{0}^{1,1}\circ\pr_{Q^{1,1}\times\Op(A'\cap B')}$ relative to $Q^{1,1}_{0}\times\Op(A'\cap B')\times\{0\}$.
\end{itemize}
Similarly, retaking $W$ if necessary, we can take a dominating $h$-spray $s_{A}:Q^{1,1}\times\Op A\times W\to Z$ over $f_{0}|_{Q^{1,1}\times\Op A}$ and a spray $\tilde s_{A}':Q^{1,1}\times\Op(A\cap B)\times W\to\C^{N}$ over $0$ such that
\begin{align*}
s^{1,1}\circ(\pr_{\Op P^{1}\times\Op(A\cap B)},\tilde s_{A}'+\tilde f_{0}^{1,1}\circ\pr_{Q^{1,1}\times\Op(A\cap B)})=s_{A}|_{Q^{1,1}\times\Op(A\cap B)\times W}.
\end{align*}
If we set $\tilde s_{A}=\tilde s_{A}'+\tilde f_{0}^{1,1}\circ\pr_{Q^{1,1}\times\Op(A\cap B)}:Q^{1,1}\times\Op(A\cap B)\times W\to\C^{N}$, the following analogous conditions hold;
\begin{itemize}
\item $s^{1,1}(\pr_{\Op P^{1}\times\Op(A\cap B)},\tilde s_{A})=s_{A}|_{Q^{1,1}\times\Op(A\cap B)\times W}$, and
\item $\tilde s_{A}$ is homotopic to $\tilde f_{0}^{1,1}\circ\pr_{Q^{1,1}\times\Op(A\cap B)}$ relative to $Q^{1,1}\times\Op(A\cap B)\times\{0\}$.
\end{itemize}
Since $(A\cap B)\cap(A'\cap B')=\emptyset$, by using $\tilde s_{A}$, $\tilde s_{B'}$ and their homotopies, we can construct a continuous map $\tilde s'_{A'}:Q^{1,1}\times\Op A'\times W\to\C^{N}$ such that
\begin{itemize}
\item $\tilde s'_{A'}|_{Q^{1,1}\times\Op(A\cap B)\times W}=\tilde s_{A},\ \tilde  s'_{A'}|_{Q^{1,1}\times\Op(A'\cap B')\times W}=\tilde s_{B'}$, and
\item $\tilde s'_{A'}(\cdot,\cdot,0)$ is homotopic to $\tilde f_{0}^{1,1}$ relative to $(Q^{1,1}\times\Op(A\cap B))\cup(Q^{1,1}_{0}\times\Op A')$.
\end{itemize}
Since $(A\cap B)\cup(A'\cap B')$ is $\cO(A')$-convex and $W$ is convex, by the Cartan--Oka--Weil theorem with parameters (cf. \cite[Theorem 2.8.4]{Forstneric2017}), we can deform $\tilde s'_{A'}$ into a family of holomorphic maps $\tilde s_{A'}:Q^{1,1}\times\Op A'\times W\to\C^{N}$ by a homotopy $H_{t}:Q^{1,1}\times\Op A'\times W\to\C^{N},\ t\in[0,1]$ such that the following hold for each $t\in[0,1]$;
\begin{itemize}
\item $H_{0}=\tilde s'_{A'},\ H_{1}=\tilde s_{A'}$,
\item $H_{t}(\cdot,\cdot,0)|_{Q^{1,1}_{0}\times\Op A'}=\tilde f_{0}^{1,1}|_{Q^{1,1}_{0}\times\Op A'}$,  and
\item $H_{t}|_{Q^{1,1}\times\Op(A\cap B)\times W}$ (resp. $H_{t}|_{Q^{1,1}\times\Op(A'\cap B')\times W}$) is a family of holomorphic maps which approximates $\tilde s_{A}$ (resp. $\tilde s_{B'}$).
\end{itemize}
If we set $s_{A'}=s^{1,1}\circ(\pr_{\Op P^{1}\times\Op A'},\tilde s_{A'}):Q^{1,1}\times\Op A'\times W\to Z$, the following conditions hold;
\begin{itemize}
\item $s_{A'}$ is close to $s_{A}$ (resp. $s_{B'}$) on $Q^{1,1}\times\Op(A\cap B)\times W$ (resp. $Q^{1,1}\times\Op(A'\cap B')\times W$), and
\item $s_{A'}(\cdot,\cdot,0)|_{Q^{1,1}_{0}\times\Op A'}=f_{0}|_{Q^{1,1}_{0}\times\Op A'}$.
\end{itemize}
Then by Forstneri\v{c}'s Heftungslemma \cite[Proposition 5.9.2, Remark 5.9.4 (C) and p.\,254]{Forstneric2017} (see also \cite[Lemma 2.4]{Kusakabe2019}), $s_{A'}$ and $s_{B'}$ amalgamate into a spray $s_{B}:Q^{1,1}\times\Op B\times rW\to Z,\ 0<r<1$ which is close to $s_{A}$ on $Q^{1,1}\times\Op(A\cap B)\times rW$ and $s_{B}(\cdot,\cdot,0)|_{Q^{1,1}_{0}\times\Op B}=f_{0}|_{Q^{1,1}_{0}\times\Op B}$.
By gluing $s_{A}$ and $s_{B}$ again, we can obtain a spray $s_{A\cup B}:Q^{1,1}\times\Op(A\cup B)\times r'W\to Z,\ 0<r'<r$ such that $s_{A\cup B}(\cdot,\cdot,0):Q^{1,1}\times\Op(A\cup B)\to Z$ is a family of holomorphic sections which satisfies
\begin{itemize}
\item $s_{A\cup B}(\cdot,\cdot,0)=f_{0}$ on $Q_{0}^{1,1}\times\Op(A\cup B)$, and
\item $s_{A\cup B}(\cdot,\cdot,0)|_{Q^{1,1}\times\Op A}$ approximates $f_{0}$ uniformly on $Q^{1,1}\times A$.
\end{itemize}
Set $t_{j}=j/lm\ (j=1,\ldots,lm)$.
Since all steps are made by homotopies which are lifted to be $\C^{N}$-valued along $s^{1,1}$ on $\Op (A\cap B)$ and $\Op (A'\cap B')$, we can obtain a homotopy $f_{t}:Q^{1,1}\times\Op(A\cup B)\to Z,\ t\in[0,t_{1}]$ which joins $f_{0}$ and $s_{A\cup B}(\cdot,\cdot,0)$ such that all intermediate families satisfy the above conditions.
By using a continuous function $\chi':Q\to[0,1]$ such that $\chi'|_{Q^{1,1}}\equiv 1$ and $\Supp\chi'$ is contained in a small neighborhood of $Q^{1,1}$ and considering $f_{\chi(q)t}(q,x)$, we may assume that we have a homotopy $f_{t}:Q\times\Op(A\cup B)\to Z,\ t\in[0,t_{1}]$ from the beginning.
Note that for each $t\in[0,t_{1}]$ the family $f_{t}$ is close to $f_{0}$ on $Q\times\Op A$.
Thus, since the global $h$-sprays in (iii) and (iv) are dominating, we can modify $s^{j,k}$ and $\tilde f_{0}^{j,k}\ ((j,k)\neq(1,1))$ slightly to satisfy the conditions (iii) and (iv) where $f_{0}$ is replaced by $f_{t_{1}}$.
Then we can move to the next step to deform $f_{t_{1}}$ on $Q^{2,1}\times\Op (A\cup B)$ and obtain a homotopy $f_{t}:Q\times\Op(A\cup B)\to Z,\ t\in[t_{1},t_{2}]$.
In this way, we deform the family inductively on $Q^{j,k}\times\Op (A\cup B)$ in the following order;
\begin{align*}
(j,k)=(1,1),(2,1),(3,1),\ldots,(l,1),(1,2),(2,2),\ldots,(l,2),(1,3),\ldots,(l,m).
\end{align*}
In each step of the induction, we set
\begin{align*}
Q_{0}^{j,k}=Q^{j,k}\cap\left(\left(P\times\left\{\tau^{k-1}\right\}\right)\cup\left(\Op\left(P_{0}\cup\bigcup_{\lambda=1}^{j-1}P^{\lambda}\right)\times\left[\tau^{k-1},1\right]\right)\right).
\end{align*}
As a result, we obtain the desired homotopy $f_{t}:Q\times\Op(A\cup B)\to Z,\ t\in[0,1]$.
\end{proof}

%
%

\section{Oka Principle for lifts: Proof of Theorem \ref{theorem:characterization}}
\label{section:lifts}

In this section, we prove the main theorem (Theorem \ref{theorem:characterization}).
The implication from $\POPAI$ to convex ellipticity follows easily from Lemma \ref{lemma:local_spray}.
The converse implication is an immediate corollary of the following stronger result (see Remark \ref{remark:pullback} for the definition of stratified convexly elliptic submersions).

\begin{theorem}
\label{theorem:oka_principle_for_lifts}
Every stratified convexly elliptic submersion enjoys $\POPAI$.
\end{theorem}

In fact, all arguments in the previous section can be applied to this case by the straightforward generalizations (cf. \cite{Forstneric2010,Forstneric2010a}) except the part where we construct dominating sprays over families (the first part of the proof of Theorem \ref{theorem:approximation}).
This part is solved by the following lemma.

\begin{lemma}
\label{lemma:parametric_spray}
Let $\pi:Y\to S$ be a stratified convexly elliptic submersion, $K$ be a Stein compact in a complex space $X$, $P$ be a topological space and $f:P\times\Op K\to Y$ be a family of holomorphic maps.
Then for any point $p\in P$ there exists a dominating global $\pi$-spray $\Op\,\{p\}\times\Op K\times\C^{N}\to Y$ over $f|_{\Op\,\{p\}\times\Op K}$.
\end{lemma}

\begin{proof}
Take an open Stein neighborhood $U\subset X$ of $K$ such that $f$ is defined on $P\times U$.
Let $\Gamma\subset U\times Y$ denote the graph of $f(p,\cdot):U\to Y$.
Since $\Gamma\cong U$ is Stein, there exists an open Stein neighborhood $V\subset U\times Y$ of $\Gamma$ by Siu's theorem \cite{Siu1976}.
Since $\pi$ is a stratified convexly elliptic submersion, the pullback submersion $(\pi\circ\pr_{Y}|_{V}\circ\pr_{V})^{*}\pi:(\pi\circ\pr_{Y}|_{V}\circ\pr_{V})^{*}Y\to V\times\C^{N}$ satisfies the assumption in Theorem \ref{theorem:oka_principle_for_sections} (see Remark \ref{remark:pullback}).
Note that holomorphic global sections of this submersion correspond to global $\pi$-sprays over $\pr_{Y}|_{V}:V\to Y$.
Thus after shrinking $V\supset\Gamma\cap(\Op K\times Y)$ if necessary, we can construct a dominating global $\pi$-spray $s:V\times\C^{N}\to Y$ over $\pr_{Y}|_{V}$ by Lemma \ref{lemma:local_spray} and Theorem \ref{theorem:oka_principle_for_sections} (in fact we do not need to shrink $V$).
Then
\begin{align*}
s\circ((\pr_{\Op K},f)\times\id_{\C^{N}}):(\Op\,\{p\}\times\Op K)\times\C^{N}\to Y
\end{align*}
is a dominating global $\pi$-spray over $f|_{\Op\,\{p\}\times\Op K}$.
\end{proof}

%
%

\section{Applications}
\label{section:applications}

In this section, we give various applications of our Oka principle.

%
%

\subsection{Equivalences between Oka properties}

It is a fundamental result in modern Oka theory that the following Oka properties of a complex manifold are equivalent:
$\CAP$, $\CIP$, $\BOPA$, $\BOPI$, $\BOPJI$, $\BOPAI$, $\BOPAJI$, $\PCAP$, $\PCIP$, $\POPA$, $\POPI$, $\POPJI$, $\POPAI$, $\POPAJI$ (see \cite[\S 5.15]{Forstneric2017} for the definitions).
As an application of Theorem \ref{theorem:characterization}, we can generalize this result to a holomorphic submersion.
To this end, we define $\CAP$ of a submersion as follows.

\begin{definition}
\label{definition:CAP}
A holomorphic submersion $\pi:Y\to S$ enjoys \emph{CAP} if for any bounded convex domain $\Omega$, any $N\in\N$, any compact convex set $K\subset\Omega\times\C^N$, any holomorphic map $F:\Omega\to S$ and any continuous map $f:\Omega\times\C^{N}\to Y$ such that $f|_{\Op K}$ is holomorphic and $\pi\circ f=F\circ\pr_{\Omega}$ there exists a holomorphic map $\tilde f:\Omega\times\C^{N}\to Y$ which approximates $f$ uniformly on $K$ and satisfies $\pi\circ\tilde f=F\circ\pr_{\Omega}$.
\end{definition}

\begin{remark}
\label{remark:CAP}
(1) A complex manifold $Y$ enjoys $\CAP$ if and only if the constant submersion $Y\to *$ enjoys $\CAP$.
\\
(2) For a holomorphic fiber bundle $\pi:Y\to S$ whose fiber enjoys $\CAP$ and its trivializing cover $\{U_{\alpha}\}_{\alpha}$, each restriction $\pi^{-1}(U_{\alpha})\to U_{\alpha}$ enjoys $\CAP$.
\end{remark}

In the same manner, we can appropriately define $\CIP$, $\PCAP$ and $\PCIP$ of a submersion.
Other Oka properties have straightforward generalizations to a submersion.
As in the case of a manifold (cf. \cite[\S5.15]{Forstneric2017}), $\CAP$ of a submersion is implied by each of other Oka properties and $\POPAJI$ implies others.
In order to prove the remaining implication $\CAP$ to $\POPAJI$, we first prove the implication from $\CAP$ to convex ellipticity.

\begin{proposition}\label{proposition:CAP=>CEll}
If a holomorphic submersion $\pi:Y\to S$ enjoys $\CAP$, then it is convexly elliptic.
\end{proposition}

\begin{proof}
Let $K\subset\C^{n}$ be a compact convex set and $f:\Op K\to Y$ be a holomorphic map.
By Lemma \ref{lemma:local_spray}, there exists a dominating local $\pi$-spray $s':\Op K\times\B^{N}\to Y$ over $f$.
Since $\pi$ enjoys $\CAP$, there exists a dominating global $\pi$-spray $s:\Op K\times\C^{N}\to Y$ which is close to $s'$ on $\Op K\times\B^{N}$.
Then Rouch\'e's theorem (cf. \cite[p.\,110]{Chirka1989}) implies the existence of a holomorphic map $\varphi:\Op K\to\C^{N}$ which is close to $0$ such that $s\circ(\id_{\Op K},\varphi)=f$.
Then $s\circ(\pr_{\Op K},\pr_{\C^{N}}+\varphi\circ\pr_{\Op K}):\Op K\times\C^{N}\to Y$ is a dominating global $\pi$-spray over $f$.
\end{proof}

We need to observe the following to obtain the implication from convex ellipticity to $\POPAJI$.

\begin{remark}
\label{remark:jet}
Since the proof of our Oka principle depends on the implication from $\HAP$ to $\POPAJI$ (cf. \cite[Theorem 6.6.6]{Forstneric2017}), we can include the jet interpolation property into our Oka principle.
That is, every stratified convexly elliptic submersion enjoys $\POPAJI$.
More generally, an Oka principle which relies on $\HAP$ (e.g. the Oka principle for sections of branched maps \cite[Theorem 6.14.4]{Forstneric2017}) holds under the appropriate assumption which imposes convex ellipticity.
\end{remark}

By Proposition \ref{proposition:CAP=>CEll} and Remark \ref{remark:jet}, we obtain the following equivalences.

\begin{corollary}
\label{corollary:equivalence}
The following Oka properties of a holomorphic submersion are equivalent:
$\CAP$, $\CIP$, $\BOPA$, $\BOPI$, $\BOPJI$, $\BOPAI$, $\BOPAJI$, $\PCAP$, $\PCIP$, $\POPA$, $\POPI$, $\POPJI$, $\POPAI$, $\POPAJI$.
\end{corollary}

As an immediate application of Corollary \ref{corollary:equivalence}, we can obtain the following invariance of the parametric Oka property (see also \cite{Forstneric2010}).
In fact, it follows also from the invariance of convex ellipticity and Theorem \ref{theorem:characterization}.

\begin{corollary}[{cf. \cite[Corollary 2.51]{Forstneric2013}}]
\label{corollary:fibration}
Let $\pi_{1}:Y_{1}\to S$, $\pi_{2}:Y_{2}\to S$ and $\pi:Y_{1}\to Y_{2}$ be holomorphic submersions and assume that $\pi$ is a surjective Oka map such that $\pi_{2}\circ\pi=\pi_{1}$.
Then $\pi_{1}$ enjoys $\POPAI$ if and only if $\pi_{2}$ enjoys $\POPAI$.
\end{corollary}

%
%

\subsection{Refinements of Oka principles}

It has been unknown whether we may take non-Euclidean parameter spaces in the following previously known Oka principles;
\begin{enumerate}
\item the Oka principle for holomorphic fiber bundles with $\CAP$ fibers (cf. \cite[Theorem 5.4.4]{Forstneric2017}),
\item the Oka principle for $\BOPAI$ submersions (cf. \cite[Theorem 7.4.3]{Forstneric2017}), and
\item the Oka principles for stratified subelliptic submersions and stratified fiber bundles with $\CAP$ fibers (cf. \cite[Corollary 7.4.5]{Forstneric2017}; see also \cite[Remark 4.6]{Forstneric2010}).
\end{enumerate}
Our Oka principle for stratified convexly elliptic submersions (Theorem \ref{theorem:oka_principle_for_lifts}) unifies these Oka principles and hence solves the above problem.

\begin{corollary}
\label{corollary:parameter}
In any of the above situations, the holomorphic submersion enjoys $\POPAI$ (without any restriction on compact Hausdorff parameter spaces)\footnote{In fact, it follows that we can also drop the Hausdorff assumption (see \cite[p.\,1160]{Larusson2015}).}.
\end{corollary}

Next, we prove the following dimensionwise Oka principle which is asked by Forstneri\v{c} (private communication).
For a fixed natural number $n\in\N$, a complex manifold enjoys $\POPAI_{n}$ if it enjoys $\POPAI$ where the dimension of a Stein space $X$ (the domain of maps in a family) is restricted as $\dim X\leq n$.

\begin{corollary}
\label{corollary:dimensionwise}
Let $Y$ be a complex manifold and $n\in\N$ be a fixed natural number.
Assume that for any compact convex set $K\subset\C^{n}$ and any holomorphic map $f:\Op K\to Y$ there exists a dominating global spray (with respect to the constant submersion) $\Op K\times\C^{N}\to Y$ over $f$.
Then $Y$ enjoys $\POPAI_{n}$.
In particular, if for any holomorphic disc $f:\D\to Y$ there exists a dominating global spray $\D\times\C^{N}\to Y$ over $f$, then $Y$ enjoys $\POPAI_{1}$ (i.e. the parametric Oka principle for maps from open Riemann surfaces holds).
\end{corollary}

The above corollary follows from the proof of the Forstneri\v{c} type Oka principle \cite[\S 5.7--5.13]{Forstneric2017} and the following approximation theorem of a local spray.

\begin{lemma}
Let $Y$ be a complex manifold, $n\in\N$ be a fixed natural number and $d$ be a distance function of $Y$.
Assume that for any compact convex set $K\subset\C^{n}$ and any holomorphic map $f:\Op K\to Y$ there exists a dominating global spray $\Op K\times\C^{N}\to Y$ over $f$.
Then for any complex manifold $X$ with $\dim X\leq n$, any very special Cartan pair $(A,B)$ in $X$, any compact Hausdorff spaces $P_{0}\subset P$ and any family of continuous map $f_{0}:Q\times\Op(A\cup B)\times\B^{N}\to Y$ where $Q=P\times[0,1],\ Q_{0}=(P\times\{0\})\cup(P_{0}\times[0,1])$ such that $f_{0}|_{Q_{0}\times\Op(A\cup B)\times\B^{N}}$ and $f_{0}|_{Q\times\Op A\times\B^{N}}$ are families of holomorphic maps, there exist a number $0<\delta<1$ such that for any $\varepsilon>0$ there exists a homotopy $f_{t}:Q\times\Op(A\cup B)\times\delta\B^{N}\to Y,\ t\in[0,1]$ such that the following hold for each $t\in[0,1]$;
	\begin{enumerate}
	\item $f_{t}=f_{0}$ on $Q_{0}\times\Op(A\cup B)\times\delta\B^{N}$,
	\item $f_{t}|_{Q\times\Op A\times\delta\B^{N}}$ is a family of holomorphic maps with $\sup_{Q\times A\times\delta\B^{N}}d(f_{t},f_{0})<\varepsilon$, and
	\item $f_{1}:Q\times\Op(A\cup B)\times\delta\B^{N}\to Y$ is a family of holomorphic maps.
	\end{enumerate}
\end{lemma}

\begin{proof}
Let $f_{0}:Q\times\Op(A\cup B)\times\B^{N}\to Y$ be a family of continuous map as above.
By assumption, for each $q\in Q$ there exists a dominating global spray $\Op(A\cap B)\times\C^{L}\to Y$ over $f(q,\cdot,0):\Op(A\cap B)\to Y$.
In the same way as in the proof of \ref{theorem:approximation}, by the compactness of $Q$, there exists a number $0<\delta<1$ such that for each $q\in Q$ there exists a global dominating spray $\Op(A\cap B)\times\delta^{1/2}\B^{N}\times\C^L\to Y$ over $f_0(q,\cdot,\cdot)|_{\Op(A\cap B)\times\delta^{1/2}\B^{N}}$.
Then Theorem \ref{theorem:approximation} for the very special Cartan pair $(A\times\delta\overline{\B^{N}},B\times\delta\overline{\B^{N}})$ implies the conclusion.
\end{proof}

%
%

\subsection{A holomorphic submersion which enjoys POPAI but is neither subelliptic nor locally trivial at any base point}

\begin{proposition}
\label{proposition:example}
Let $\{a_j\}_{j\in\N}$ be a dense sequence in $\C$ and set
\begin{align*}
\Sigma=\{(a_j,j,0)\}_{j\in\N}\times\left(\overline{\N^{-1}}\right)^2\subset\C^5
\end{align*}
where $\N^{-1}=\{j^{-1}:j\in\N\}$.
Then the submersion $\pi:\C^5\setminus\Sigma\to\C,\ (z_{1},\ldots,z_{5})\mapsto z_{1}$ enjoys $\POPAI$ but for any nonempty open set $U\subset\C$ the restriction $\pi^{-1}(U)\to U$ is neither trivial nor subelliptic.
\end{proposition}

The following lemma is an immediate consequence of Theorem \ref{theorem:characterization} (in fact we can prove it without using Theorem \ref{theorem:characterization}).

\begin{lemma}
\label{lemma:increasing_union}
Let $\pi:Y\to S$ be a holomorphic submersion.
Assume that for any compact set $K\subset Y$ there exists an open neighborhood $U\subset Y$ of $K$ such that the restriction $\pi|_U:U\to S$ enjoys $\POPAI$.
Then $\pi:Y\to S$ enjoys $\POPAI$.
\end{lemma}

\begin{proof}[Proof of Proposition \ref{proposition:example}]
Note that for each $j\in\N$ the fiber $\pi^{-1}(a_j)\cong\C^4\setminus(\{0\}^2\times(\overline{\N^{-1}})^2)$ is not subelliptic by \cite[Corollary 3.2]{Kusakabe2020}.
On the other hand, $\pi^{-1}(z)\cong\C^4$ for each $z\not\in\{a_j\}_{j\in\N}$.
These imply the latter statement.

Let us prove that the submersion $\pi:\C^{5}\setminus\Sigma\to\C$ enjoys $\POPAI$ by using Lemma \ref{lemma:increasing_union}.
Take an arbitrary compact set $K\subset\C^{5}\setminus\Sigma$.
Then there exists $k\in\N$ such that
\begin{align*}
U=\C^{5}\setminus\left(\left(\{(a_j,j,0)\}_{j=1}^{k}\times\left(\overline{\N^{-1}}\right)^2\right)\cup\left(\C\times\{(j,0)\}_{j=k+1}^{\infty}\times\C^{2}\right)\right)
\end{align*}
contains $K$.
Consider the restriction $\pi|_{U}:U\to\C$.
Note that
\begin{align*}
\pi|_{U}^{-1}(a_{l})=\C^{4}\setminus\left(\left(\{(l,0)\}\times\left(\overline{\N^{-1}}\right)^2\right)\cup\left(\{(j,0)\}_{j=k+1}^{\infty}\times\C^{2}\right)\right)
\end{align*}
for each $l=1,\ldots,k$.
On the other hand, for each $z\not\in\{a_j\}_{j=1}^{k}$ the fiber $\pi_{U}^{-1}(z)=(\C^2\setminus\{(j,0)\}_{j=k+1}^{\infty})\times\C^{2}$ is Oka (cf. \cite[Proposition 5.6.17]{Forstneric2017}).
Since a stratified holomorphic fiber bundle with Oka fibers enjoys $\POPAI$ (the Forstneri\v{c} type Oka principle), it suffices to prove that for each $l=1,\ldots,k$ the fiber $\pi|_{U}^{-1}(a_{l})$ is Oka.
Consider the projection $\pi_{1,2}:\pi|_{U}^{-1}(a_{l})\to\C^{2}\setminus\{(j,0)\}_{j=k+1}^{\infty},\ (z_{1},z_{2},z_{3},z_{4})\mapsto (z_{1},z_{2})$.
All fibers coincide with $\C^{2}$ except $\pi_{1,2}^{-1}((l,0))=\C^{2}\setminus(\overline{\N^{-1}})^2$.
By using an automorphism $\varphi\in\Aut\C^{2}$ which fixes $\{(j,0)\}_{j=k+1}^{\infty}$ and satisfies $\varphi(l,0)=(l,1)$, it can be seen that each point of $\pi|_{U}^{-1}(a_{l})$ has a Zariski open\footnote{A subset of a complex space is said to be \emph{Zariski open} if its complement is a closed complex subvariety.} neighborhood of the form
\begin{align*}
\varphi^{-1}\left(\left(\C\times\C^{*}\times\C^{2}\right)\setminus\left(\{(l,1)\}\times\left(\overline{\N^{-1}}\right)^2\right)\right).
\end{align*}
By the localization principle for Oka manifolds \cite[Theorem 1.4]{Kusakabe2019}, it suffices to prove such a Zariski open set is Oka.
Since
\begin{align*}
\C^{4}\setminus\left(\{l\}\times\Z\times\left(\overline{\N^{-1}}\right)^2\right)&\to\left(\C\times\C^{*}\times\C^{2}\right)\setminus\left(\{(l,1)\}\times\left(\overline{\N^{-1}}\right)^2\right), \\
(z_{1},z_{2},z_{3},z_{4})&\mapsto(z_{1},\exp(2\pi iz_{2}),z_{3},z_{4})
\end{align*}
is a covering map, it suffices to prove that $\C^{4}\setminus(\{l\}\times\Z\times(\overline{\N^{-1}})^2)$ is Oka (see Corollary \ref{corollary:fibration}).
Consider the Zariski open covering of $\C^{4}\setminus(\{l\}\times\Z\times\{(0,0)\})=(\{l\}\times\Z\times\{(0,0)\})^{c}$ by $U_{1}=(\C^{2}\times\{0\}\times\C)^{c}$, $U_{2}=(\C^{3}\times\{0\})^{c}$ and $U_{3}=(\{l\}\times\Z\times\C^{2})^{c}$.
Note that the following is a Zariski open covering
\begin{align*}
\C^{4}\setminus\left(\{l\}\times\Z\times\left(\overline{\N^{-1}}\right)^2\right)=\bigcup_{j=1}^{3}\left(U_{j}\setminus\left(\{l\}\times\Z\times\left(\overline{\N^{-1}}\right)^2\right)\right).
\end{align*}
Since $U_{3}\cap(\{l\}\times\Z\times(\overline{\N^{-1}})^2)=\emptyset$, the Zariski open set $U_{3}\setminus(\{l\}\times\Z\times(\overline{\N^{-1}})^2)=U_{3}=(\C^{2}\setminus(\{l\}\times\Z))\times\C^{2}$ is Oka (cf. \cite[Proposition 5.6.17]{Forstneric2017}).
On the other hand, we can see that $U_{1}\setminus(\{l\}\times\Z\times(\overline{\N^{-1}})^2)\cong U_{2}\setminus(\{l\}\times\Z\times(\overline{\N^{-1}})^2)$ are Oka because they are universally covered by $\C^{4}\setminus(\{l\}\times\Z\times\exp^{-1}(\N^{-1})\times\overline{\N^{-1}})\cong\C^{4}\setminus(\{l\}\times\Z\times\overline{\N^{-1}}\times\exp^{-1}(\N^{-1}))$ which are Oka by \cite[Theorem 1.2]{Kusakabe2020}.
By the localization principle \cite[Theorem 1.4]{Kusakabe2019} again, $\C^{4}\setminus(\{l\}\times\Z\times(\overline{\N^{-1}})^2)$ is Oka.
\end{proof}

\begin{remark}
\label{remark:Stein}
For a holomorphic submersion $\pi:Y\to S$ from a Stein space $Y$, it can be easily seen that ellipticity, subellipticity and $\POPAI$ are equivalent.
Indeed, if $\pi$ enjoys $\POPAI$ then there exists a dominating global $\pi$-spray $Y\times\C^{N}\to Y$ over the identity map $\id_{Y}$ (see \cite[Proof of Corollary 8.8.7]{Forstneric2017}), hence $\pi$ is elliptic.
\end{remark}

%
%

\subsection{Localization principle for Oka maps}

In \cite{Kusakabe2019}, we proved the localization principle for Oka manifolds as an application of the characterization of Oka manifolds by convex ellipticity.
The same proof shows the following localization principle.

\begin{corollary}
\label{corollary:localization}
Let $\pi:Y\to S$ be a holomorphic submersion.
Assume that for any point $p\in Y$ there exists a Zariski open neighborhood $U\subset Y$ of $p$ such that the restriction $\pi|_{U}:U\to S$ enjoys $\POPAI$.
Then $\pi$ enjoys $\POPAI$.
\end{corollary}

By using this localization principle, we give new examples of Oka maps.
In \cite{Hanysz2014}, Hanysz studied the Oka property of the graph complement $(X\times\P^{1})\setminus(\Gamma_{f}\cup\Gamma_{g})$ for a holomorphic map $f:X\to\P^{1}$ and the constant map $g\equiv\infty$.
Under a certain assumption on $f$, he proved that the graph complement is Oka if and only if $X$ is Oka \cite[Theorem 4.6]{Hanysz2014} and that the projection $(X\times\P^{1})\setminus(\Gamma_{f}\cup\Gamma_{g})\to X$ is elliptic and hence enjoys $\POPAI$ \cite[Remark 4.12]{Hanysz2014}.
The latter statement holds even if we consider any nonconstant holomorphic maps $f,g:X\to\P^{1}$ since the projection is stratified elliptic.
We have the following generalization of this statement.
It also generalizes our previous result \cite[Corollary 1.5]{Kusakabe2019} (see also Corollary \ref{corollary:configuration}).

\begin{corollary}
\label{corollary:graph_complement}
Let $Y$ be a complex manifold of dimension at least two.
Assume that $Y$ is Zariski locally biholomorphic to some abelian complex Lie group.
Then for any complex space $X$ and any holomorphic maps $f_{1},\ldots,f_{k}:X\to Y$ the projection $\pr_{X}:(X\times Y)\setminus\bigcup_{j=1}^{k}\Gamma_{f_{j}}\to X$ from the graph complement enjoys $\POPAI$.
This holds in particular for any smooth toric variety $Y$ of dimension at least two.
\end{corollary}

\begin{proof}
By the localization principle (Corollary \ref{corollary:localization}), it suffices to prove that the restriction $\pr_{X}:(X\times U)\setminus\bigcup_{j=1}^{k}\Gamma_{f_{j}}\to X$ enjoys $\POPAI$ where $U\subset Y$ is a Zariski open subset which is biholomorphic to some abelian complex Lie group.
After shrinking $U$ if necessary, we may assume that this abelian complex Lie group is connected and hence there exists a lattice $\Lambda\subset\C^{n}$ (i.e. a discrete subgroup) such that $U\cong\C^{n}/\Lambda$. 
Let us stratify $X$ by closed complex subvarieties
\begin{align*}
X_{l}=\left\{x\in X:\#\left(\{f_{j}(x)\}_{j=1}^{k}\cap U\right)\leq l\right\},\quad l=-1,0,\ldots,k.
\end{align*}
By considering each stratum of $X$, we may assume that $f_{1},\ldots,f_{k}\in\cO(X,U)$ and $\Gamma_{f_{j}}\cap\Gamma_{f_{l}}=\emptyset$ if $j\neq l$ from the beginning.
In the following, we identify $U$ with $\C^{n}/\Lambda$.
By using the universal covering map $\pi:\C^{n}\to\C^{n}/\Lambda$ and Corollary \ref{corollary:fibration}, we can see that $\pr_{X}:(X\times(\C^{n}/\Lambda))\setminus\bigcup_{j=1}^{k}\Gamma_{f_{j}}\to X$ enjoys $\POPAI$ if and only if $\pr_{X}:(X\times\C^{n})\setminus\bigcup_{j=1}^{k}(\id_{X}\times\pi)^{-1}(\Gamma_{f_{j}})\to X$ enjoys $\POPAI$.
Take an arbitrary point $p\in X$.
By the results of Buzzard \cite[Proofs of Theorem 4.2 and Theorem 4.3]{Buzzard2003}, there exist mutually disjoint open neighborhoods $U_{j}\subset\C^{n}/\Lambda\ (j=1,\ldots,k)$ of $f_{j}(p)$ and an automorphism $\varphi\in\Aut\C^{n}$ such that
\begin{align*}
\varphi\left(\bigcup_{j=1}^{k}\pi^{-1}(U_{j})\right)\subset\{(z',z_{n})\in\C^{n}:|z_{n}|+1\leq\|z'\|\}.
\end{align*}
If we set $U=\bigcap_{j=1}^{k}f_{j}^{-1}(U_{j})\subset X$, it is an open neighborhood of $p$ which satisfies
\begin{align*}
(\id_{U}\times\varphi)\left(\bigcup_{j=1}^{k}(\id_{U}\times\pi)^{-1}(\Gamma_{f_{j}|_{U}})\right)\subset U\times\{(z',z_{n})\in\C^{n}:|z_{n}|+1\leq\|z'\|\}.
\end{align*}
Thus $\bigcup_{j=1}^{k}(\id_{X}\times\pi)^{-1}(\Gamma_{f_{j}})\subset X\times\C^{n}$ is locally uniformly tame (see \cite[Proposition 6.4.14]{Forstneric2017} for the definition of locally uniformly tame subvarieties).
Therefore $\pr_{X}:(X\times\C^{n})\setminus\bigcup_{j=1}^{k}(\id_{X}\times\pi)^{-1}(\Gamma_{f_{j}})\to X$ is an elliptic submersion (cf. \cite[Proposition 6.4.14]{Forstneric2017}) and hence enjoys $\POPAI$.
\end{proof}

Corollary \ref{corollary:graph_complement} particularly implies the following Oka principle.

\begin{corollary}
Let $X$ be a Stein space, and let $Y$ and $f_{1},\ldots,f_{k}$ be as in Corollary \ref{corollary:graph_complement}.
If there exists a continuous map $f:X\to Y$ such that $f(x)\neq f_{j}(x)$ for all $j=1,\ldots,k$ and $x\in X$, then there exists a holomorphic map $\tilde f:X\to Y$ such that $\tilde f(x)\neq f_{j}(x)$ for all $j=1,\ldots,k$ and $x\in X$.
\end{corollary}

For a complex manifold $Y$, let us consider the \emph{configuration space} of ordered $n$-tuples of points in $Y$: $F(Y,n)=\{(y_{1},\ldots,y_{n})\in Y^{n}:y_{j}\neq y_{k}\ \mbox{if}\ j\neq k\}$.
Kutzschebauch and Ramos-Peon \cite[Theorem 3.1]{Kutzschebauch2017} proved that for a Stein manifold $Y$ of dimension at least two with the density property or the volume density property its configuration space $F(Y,n)$ is Oka for each $n\in\N$.
As an immediate consequence of Corollary \ref{corollary:graph_complement}, we obtain the following analogous result for abelian complex Lie groups and smooth toric varieties.

\begin{corollary}
\label{corollary:configuration}
Let $Y$ be a complex manifold of dimension at least two.
Assume that $Y$ is Zariski locally biholomorphic to some abelian complex Lie group.
Then for each $n\in\N$ the projection $F(Y,n+1)\to F(Y,n),\ (y_{1},\ldots,y_{n+1})\mapsto(y_{1},\ldots,y_{n})$ is an Oka map.
In particular, the following hold:
\begin{enumerate}
\item For any complex manifold $X$ and any holomorphic maps $f_{1},\ldots,f_{k}:X\to Y$ whose graphs are mutually disjoint, the graph complement $(X\times Y)\setminus\bigcup_{j=1}^{k}\Gamma_{f_{j}}$ is Oka if and only if $X$ is Oka.
\item For each $n\in\N$, the configuration space $F(Y,n)$ is Oka.
\item For any finite subset $A\subset Y$, the complement $Y\setminus A$ is Oka.
\end{enumerate}
\end{corollary}

It is natural to ask the following problem.
An affirmative answer to it solves the problem whether a punctured Oka manifold of dimension at least two is Oka \cite[Problem 7.6.2]{Forstneric2017}.

\begin{problem}
Let $Y$ be an Oka manifold of dimension at least two.
Is the projection $F(Y,n+1)\to F(Y,n),\ (y_{1},\ldots,y_{n+1})\mapsto(y_{1},\ldots,y_{n})$ an Oka map for each $n\in\N$?
\end{problem}

%
%

\section*{Acknowledgement}
I would like to thank my supervisor Katsutoshi Yamanoi, Finnur L\'{a}russon and Luca Studer for encouragement and many useful comments.
I also wish to thank Franc Forstneri\v{c} for his question on Corollary \ref{corollary:dimensionwise} which led to this paper, and for helpful discussions and his hospitality during my visit at University of Ljubljana.
The idea of approximating $\tilde s_{A}$ in the proof of Theorem \ref{theorem:approximation} is due to him.
This work was supported by JSPS KAKENHI Grant Number JP18J20418.

%
%

\end{document}